\documentclass[11pt,a4paper]{amsart}

\usepackage{amsmath, amssymb, amscd}
\usepackage{a4wide,bbm}
\usepackage{graphicx}

\newtheorem{thm}{Theorem}
\newtheorem{fact}[thm]{Fact}
\newtheorem{lem}[thm]{Lemma}

\theoremstyle{definition}

\newtheorem{rem}[thm]{Remark}

\newcommand{\ts}{\hspace{0.5pt}}
\newcommand{\nts}{\hspace{-0.5pt}}

\newcommand{\RR}{\mathbb{R}\ts}
\newcommand{\QQ}{\mathbb{Q}\ts}
\newcommand{\ZZ}{\mathbb{Z}}
\newcommand{\TT}{\mathbb{T}}
\newcommand{\NN}{\mathbb{N}}
\newcommand{\XX}{\mathbb{X}\ts}
\newcommand{\YY}{\mathbb{Y}}
\newcommand{\CC}{\mathbb{C}}

\newcommand{\cA}{\mathcal{A}}
\newcommand{\cB}{\mathcal{B}}
\newcommand{\cE}{\mathcal{E}}
\newcommand{\cG}{\mathcal{G}}

\newcommand{\cS}{\mathcal{S}}
\newcommand{\cT}{\mathcal{T}}
\newcommand{\cR}{\mathcal{R}}

\newcommand{\cO}{\mathcal{O}}
\newcommand{\cU}{\mathcal{U}}

\newcommand{\xp}{x^{\ts \prime}}
\newcommand{\yp}{y^{\ts \prime}}

\newcommand{\one}{\mathbbm{1}}

\newcommand{\Aut}{\mathrm{Aut}}

\newcommand{\cent}{\mathrm{cent}}
\newcommand{\norm}{\mathrm{norm}}

\newcommand{\id}{\mathrm{Id}}

\newcommand{\GA}{\mathrm{GA}}
\newcommand{\GL}{\mathrm{GL}}
\newcommand{\PGL}{\mathrm{PGL}}
\newcommand{\Mat}{\mathrm{Mat}}

\newcommand{\bs}{\boldsymbol}

\begin{document}

\title[Reversing and extended symmetries of dynamical systems]{A 
 brief guide to reversing and extended\\[2mm] symmetries
  of dynamical systems}

\author{Michael Baake}
\address{Faculty of Mathematics, Bielefeld University, 
  Box 100131, 33501 Bielefeld, Germany}
\email{mbaake@math.uni-bielefeld.de}

\begin{abstract}
  The reversing symmetry group is a well-studied extension of the
  symmetry group of a dynamical system, the latter being defined by
  the action of a single homeomorphism on a topological space. While
  it is traditionally considered in nonlinear dynamics, where the
  space is simple but the map is complicated, it has an interesting
  counterpart in symbolic dynamics, where the map is simple but the
  space is not.  Moreover, there is an interesting extension to the
  case of higher-dimensional shifts, where a similar concept can be
  introduced via the centraliser and the normaliser of the acting
  group in the full automorphism group of the shift space.  We recall
  the basic notions and review some of the known results, in a fairly
  informal manner, to give a first impression of the phenomena that
  can show up in the extension from the centraliser to the normaliser,
  with some emphasis on recent developments.
\end{abstract}

\maketitle

\section{Introduction}\label{sec:intro}

Symmetries of dynamical systems are important objects to study, as
they help in understanding the orbit structure and many other
properties. Moreover, the group of symmetries is a topological
invariant that can be useful for distinguishing between different
dynamical systems. Naturally, this invariant is generally weaker than
other invariants (such as those from (co{\ts}-)homology or homotopy
theory), but often easier to access.

For both aspects, studying properties and defining invariants, one is
clearly interested in effective generalisations or extensions of the
symmetry group.  Inspired by the \mbox{time{\ts}-}reversal symmetry of
many fundamental equations in physics, one obvious step in this
direction is provided by the \emph{reversing symmetry group} of a
dynamical system, which{\ts}---{\ts}in the case of
reversibility{\ts}---{\ts}is an index-$2$ extension of the symmetry
group.

Traditionally, the majority of the studies has concentrated on
concrete dynamical systems, where the space is usually simple, but the
mapping(s) might be complicated.  Even for toral automorphism, the
answer is amazingly rich.  There is a complementary picture, which
arises through the coding of itineraries and leads to the analogous
questions in symbolic dynamics \cite{MH}. Here, the mapping(s) are
simple, but the space (usually a closed shift space) is complicated,
and this is particularly so when going to higher-dimensional shifts.

In this brief introductory review, we recall the basic definitions and
notions, and present some results from the large body of literature
that has accumulated. Clearly, the exposition cannot be complete in
any way, whence the references will provide further directions.

After some examples from the classic theory of concrete dynamical
systems, we shall stroll through some more recent results on the
complementary picture from symbolic dynamics.

\section{General setting and notions}\label{sec:prelim}

A convenient starting point is a topological space $\XX$, which is
usually (but not always) assumed to be compact, and a mapping
$T \in \Aut (\XX)$, where the automorphism group is understood in the
Smale sense, meaning that it is the group of \emph{all} homeomorphisms
of $\XX$. The pair $(\XX, T)$ then defines a (topological)
\emph{dynamical system}, and the group
$\langle T \rangle \subset \Aut (\XX)$ is important. Now, we define
the \emph{symmetry group} of $(\XX, T)$ as
\begin{equation}\label{eq:def-S}
  \cS (\XX,T) \, := \, \{ G \in \Aut (\XX) : 
     G \circ T = T \circ G \} \, = \,
  \cent^{}_{\Aut (\XX)} ( \langle T \rangle ) 
     \, = \, \Aut (\XX,T) \ts .
\end{equation}
This group plays an important role in the analysis of $(\XX, T)$, for
instance in the context of periodic orbits and dynamical zeta
functions. Its is also a useful tool in the classification of
dynamical systems, because it is a topological invariant.

\begin{rem}
  The group $\Aut (\XX,T)$ is often used as a starting point for
  algebraic considerations, and then simply called the automorphism
  group of the dynamical system, but this is{\ts}---{\ts}as we shall
  see later {\ts}---{\ts}a use of the word that is too restrictive,
  and effectively excludes many natural mappings from the
  consideration.  We will thus not use this notation, and rather view
  $\cS (\XX,T)$ as a subgroup of $\Aut (\XX)$ in the Smale sense.

  In some cases, the group $\Aut (\XX)$ might be too big a `universe'
  to consider, and some subgroup of it is a more natural choice, for
  instance when some additional structure of $\XX$ should be
  preserved. This is particularly so if some general results are
  available that imply $\cS (\XX,T)$ and $\cR (\XX,T)$ to be subgroups
  of some $\,\cU \subset \Aut (\XX)$. In this case, one can start with
  $\cU$, and simplify the algebraic derivations considerably. The
  above point simply is that $\cU$ should generally \emph{not} be
  chosen as $\Aut (\XX, T)$, as this is too restrictive.
\end{rem}

Since we will not consider the case that $T$ is not invertible, a
natural extension of $\cS (\XX,T)$ is given by
\begin{equation}\label{eq:def-R}
  \cR (\XX,T) \, := \, \{ G \in \Aut (\XX) : 
  G \circ T \circ G^{-1} = T^{\pm 1} \} \ts ,
\end{equation}
which is motivated by the \mbox{time{\ts}-}reversal symmetries of many
fundamental equations of physics; see \cite{Sev,LR} and references
therein for background. From now on, we write $G \ts T$ instead of
$G\circ T$ etc.\ for ease of notation.  The relation between
$\cS (\XX,T)$ and $\cR (\XX,T)$ can be summarised as follows; see
\cite{BR-Bulletin} and referenced therein.

\begin{thm}
  Let\/ $(\XX, T)$ be a topological dynamical system. Then,
  $\cR (\XX,T)\subset \Aut (\XX)$ is a group, with\/
  $\langle T \rangle$ and\/ $\cS (\XX,T)$ as normal subgroups. One
  either has\/ $\cR (\XX,T) = \cS (\XX,T)$ or\/
  $[\cR (\XX,T) : \cS (\XX,T)]=2$. In the latter case, the systems is
  reversible.
      
  Moreover, if\/ $T^2 \ne \id$ and if there is an involution\/ $H$
  with\/ $HTH^{-1} = T^{-1}$, one has
\[
     \cR (\XX,T) \, = \, \cS (\XX,T) \rtimes \langle H \rangle
     \, \simeq \, \cS (\XX,T) \rtimes C_{2} \ts ,
\]   
  which is the standard form of reversibility.
\end{thm}

An element that conjugates $T$ into its inverse (where we assume
$T^2 \ne \id$) is called a \emph{reversor}.  An elementary observation
is the fact that a reversor cannot be of odd order, so it is either of
even or of infinite order. When the order is finite, hence of the form
$2^\ell (2m+1)$ for some $\ell\geqslant 1$, there exists another
reversor of order $2^\ell$.  When $T$ possesses an involutory
reversor, $R$ say, one has $T=T R^2 = (TR)R$, where
$(TR)^2 = TR \ts\ts TR = T \ts T^{-1}=\id$, so $T$ is the product of
two involutions. This is a frequently used approach in the older
literature, before the group-theoretic setting showed \cite{Lamb,G1}
that the more general approach is natural and helpful; see
\cite{LR,BR-Bulletin} and references therein for details.

As is implicit from our formulation so far, reversibility is not an
interesting concept when $T$ itself is an involution. More generally,
when $T$ has finite order, the structure of $\cR (\XX,T)$ is a
group-theoretic problem, and of independent interest; see \cite{OS}
for a concise exposition. However, in the context of dynamical
systems, one is mainly interest in the case that
$\langle T \rangle \simeq \ZZ$. Then, one can slightly change the
point of view by considering $T$ as defining a continuous group action
of $\ZZ$ on $\XX$, which is often reflected by the modified notation
$(\XX, \ZZ)$ for the topological dynamical system. From now on, unless
explicitly stated otherwise, we shall adopt this point of view here,
too. The following result is elementary.

\begin{fact}
  When\/ $T$ is not of finite order, one has\/
  $\cR (\XX,T) = \norm^{}_{\Aut (\XX)} ( \langle T \rangle )$.
\end{fact}

It is thus the interplay between the (topological) centraliser and
normaliser that is added in the extension from $\cS (\XX,T)$ to
$\cR (\XX,T)$. One simple (but frequently useful) instance of this is
given by the following result, where $C_\infty$ and
$D_\infty = C_\infty \rtimes C^{}_{2}$ denote the infinite cyclic and
dihedral group, respectively.

\begin{thm}[{\cite[Thm.~1 and Cor.~1]{BR-Bulletin}}]
  Let\/ $T \in \Aut (\XX)$ be of infinite order. If\/
  $\cS (\XX,T) \simeq C_{\infty}$ and if\/ $T$ is reversible, one has\/
  $\cR (\XX,T) = \cS (\XX,T) \rtimes C_{2} \simeq D_{\infty}$, and all
  reversors of\/ $T$ are involutions.
   
  Conversely, if all reversors of\/ $T$ are involutions, the symmetry
  group\/ $\cS (\XX,T)$ is Abelian.
\end{thm}

Clearly, in the setting of dynamical systems, one could equally well
consider the analogous questions for the \mbox{measure{\ts}-}theoretic
centraliser and normaliser, and this is indeed frequently done in the
literature; compare \cite{GLR,K-book,FKL} and references therein.
Since, in many relevant cases, the \mbox{measure{\ts}-}theoretic
symmetry groups turn out to be topological (see \cite{K-book} for
results in this direction), we concentrate on the latter situation in
this overview.

In what follows, we shall meet two rather different general
situations, as briefly indicated in the introduction. On the one hand,
there are many systems from nonlinear dynamics where the space is
simple, but the map is complicated. In this case, we will write
$\cS (T)$ instead of $\cS (\XX,T)$ to emphasise the mapping. Likewise,
when we are in the complementary situation (of symbolic dynamics, say)
with a simple map acting on a more complicated space, we will use
$\cS (\XX)$ instead to highlight the difference. This also matches the
widely used conventions in these two directions.

\section{Concrete systems from nonlinear dynamics}

In this section, we will describe, in a somewhat informal manner,
how symmetries and reversing symmetries arise in three particular
families of dynamical systems, namely trace maps, toral
automorphisms, and polynomial autormorphisms of the plane.
Clearly, there are many other relevant examples, some of which
can be found in \cite{Sev,LR,OS} and references therein.

\subsection{Trace maps}

This class of dynamical system arises in the study of
\mbox{one{\ts}-}dimensional Schr\"{o}dinger operators with aperiodic
potentials of substitutive origin, compare \cite{DY} and references
therein, and provide a powerful tool for the study of their spectra
and transport properties. The paradigmatic \emph{Fibonacci trace map}
in $3$-space is given by
\[
     (x,y,z) \, \longmapsto \, (y,z,2\ts yz - x)
\]
and is reversible, with involutory reversor $(x,y,z)\mapsto (z,y,x)$;
see \cite{RB} and references given there. The group{\ts}-theoretic
`universe' to consider here is given by the group of $3$-dimensional
invertible polynomial mappings that preserve the
Fricke{\ts\ts}--{\nts}Vogt invariant
\[
    I (x,y,z) \, = \, x^2 + y^2 + z^2 - 2 \ts x y z - 1
\]
and fix the point $(1,1,1)$; see \cite{RB,BGJ,BR-trace} and references
therein for more. This group of mappings is isomorphic with
$\PGL (2,\ZZ)$, and can thus be analysed by classic methods, including
the theory of binary quadratic forms.

In other words, the analysis of (reversing) symmetries of trace maps
is equivalent to the determination of $\cS (M)$ and $\cR(M)$ for
matrices $M\in \PGL (2,\ZZ)$. Since
\[
   \PGL(2,\ZZ) \, \simeq \, \GL(2,\ZZ)/\{\pm \one\} \ts ,
\]  
the following result is obvious.

\begin{fact}
   Let\/ $M\in\PGL(2,\ZZ)$ and\/ $M'$ be either of the 
   two corresponding matrices in\/ $\GL(2,\ZZ)$. 
   Then, the symmetry group\/ $\cS (M)$ is given by
\[
    \cS (M) \, = \, \cent^{}_{\PGL(2,\ZZ)} ( \langle M 
      \rangle ) \, = \, \cent^{}_{\GL(2,\ZZ)} ( \langle M'
      \rangle ) / \{ \pm \one\} \ts .
\]   
\end{fact}

The symmetry groups can thus be derived from the analysis of general
(two-dimensional) toral automorphisms, which we will review in
Section~\ref{sec:toral}.  For the reversing symmetry group, the role
of $\{ \pm \one\}$ changes. Let $M\in\PGL(2,\ZZ)$ be given, and view
it as a $\GL(2,\ZZ)$-matrix. Then, we have to find all solutions $H$
to
\[
      H M H^{-1} \, = \, \pm M^{-1} ,
\]
where the calculation modulo $\pm \one$ means that we get
\emph{more} cases with reversibility than in $\GL(2,\ZZ)$. For
instance, $M=\left(\begin{smallmatrix} 1 & 1 \\ 1 & 0
\end{smallmatrix}\right)$ is reversible in $\PGL (2,\ZZ)$, 
with an involutory reversor, but not within $\GL (2,\ZZ)$,
while $M^2$ (Arnold's cat map \cite[Ex.~1.15]{AA}) is 
reversible in both groups. Within $\GL (2,\ZZ)$,
this phenomenon is called $2$-reversibility; see 
\cite{BR-trace} for details.

\subsection{Toral automorphisms}\label{sec:toral}

These systems, which are also known as `cat maps', are much studied
examples in chaotic dynamics and ergodic theory. Here, in order to
preserve the linear structure of $\TT^d$, the $d$-dimensional torus,
one usually works within $\cU = \GL(d,\ZZ) \subset \Aut(\TT^d)$; see
\cite{AA,AP,AW,Pet} for background. 

In the planar case ($d=2$), one thus has to deal with the group
$\GL (2,\ZZ)$. Here, if $M$ is an element of infinite order, one
always finds $\cS (M) \simeq C_{2} \times C_{\infty}$, where
$C_{2} = \{ \pm \one \}$. This follows for any parabolic element by a
simple calculation, and, for the hyperbolic elements, is a consequence
of Dirichlet's unit theorem for real quadratic number fields; see
\cite{BS-Book} for background.

\begin{rem}\label{rem:types}
  Reversible cases among elements of infinite order are of three
  possible types: When all reversors are involutions, one has
  $\cR (M) \simeq C_{2} \times D_{\infty}$, with
  $D_{\infty} = C_{\infty} \rtimes C_{2}$ as before; when all
  reversors are of fourth order, one has
  $\cR (M) \simeq C_{\infty} \rtimes C_{4}$; finally, when reversors
  both of order $2$ and $4$ exist, one has
  $\cR (M) \simeq (C_{2} \times C_{\infty}) \rtimes C_{2}$. All three
  types occur; see \cite[Thm.~2 and Ex.~4]{BR-Bulletin} and references
  given there for more.
\end{rem}

In this context, it is certainly a valid and interesting question how
the concepts can be extended to cover toral endomorphisms, or what
happens when one restricts to rational sublattices. This is connected
with looking at the related questions over finite fields and residue
class rings; see \cite{BNR,BRW} and references therein for some
results. \smallskip

The situation becomes more complex, and also more interesting, in
higher dimensions. In a first step, one has to analyse the symmetry
group of a toral automorphism, $M \in\GL (d,\ZZ)$ say, within this
matrix group. In the generic case, where $M$ is simple (meaning that
its eigenvalues are distinct) one can employ Dirichlet's unit theorem
again.  Let us first look at the case that the characteristic
polynomial $P(x) = \det (M-x \one)$ of $M$ is irreducible over $\ZZ$,
and hence also over $\QQ$.  Then, if $\lambda$ is any of the $d$
eigenvalues of $M$, it is an algebraic integer of degree
$d=n^{}_{1} + 2 \ts n^{}_{2}$, where $n^{}_{1}$ is the number of real
algebraic conjugates of $\lambda$ and $n^{}_{2}$ the number of complex
conjugate pairs among the algebraic conjugates.

Now, if $\cO$ is the maximal order in the algebraic number field
$\QQ (\lambda)$, Dirichlet's unit theorem states that the unit
group $\cO^\times$ is of the form
\begin{equation}\label{eq:Diri}
      \cO^\times \, \simeq \; T \nts\times \ZZ^{n^{}_{1} + n^{}_{2} - 1}
\end{equation}
with $T = \cO \cap \{ \text{roots of unity} \}$ being a finite cyclic 
group. The latter is known as the \emph{torsion subgroup} of 
$\cO^\times$. Due to the isomorphism of $\ZZ[\lambda]$ with
the ring $\ZZ[M]$ under our irreducibility  assumption on $P$,
one then has the following result \cite[Prop.~1 and Cor.~1]{BR-toral}.

\begin{thm}\label{thm:tor-1}
  Let\/ $M\in\GL (d,\ZZ)$ have an irreducible characteristic
  polynomial, $P(x)$, of degree\/ $d = n^{}_{1} + 2 \ts n^{}_{2}$,
  with\/ $n^{}_{1}$ and\/ $n^{}_{2}$ as above.  Then, $\cS (M)$ is
  isomorphic with a subgroup of\/ $\cO^\times$ of maximal rank, so
\[
     \cS (M) \, \simeq \; T' \nts \times \ZZ^{n^{}_{1} + n^{}_{2} - 1} ,
\]    
  where\/ $T'$ is a subgroup of the torsion group\/ $T$ from
  Eq.~\eqref{eq:Diri}.
  
  Moreover, whenever\/ $P(x)$ has a real root, which includes
  all cases with\/ $d$ odd, one simply has\/ $T' = \{ \pm 1 \} 
  \simeq C^{}_{2}$.
\end{thm}

For our previous example, $M=\left(\begin{smallmatrix} 1 & 1 \\
1 & 0 \end{smallmatrix}\right)$, one finds $\cS (M) = \{ \pm \one \}
\times \nts \langle M \ts \rangle \simeq C_2 \times \ZZ$. Note that, in
general, the generators of the free part of $\cS (M)$ can correspond
to powers of fundamental units, which is related with the question of
the existence of matrix roots within $\GL (d,\ZZ)$; see
\cite{BR-toral} for more. One can quite easily extend
Theorem~\ref{thm:tor-1} to the case that $M$ is simple.  This is done
by factoring $P$ over $\ZZ$ and treating the factors separately
\cite[Thm.~1]{BR-trace}. \smallskip

Let us look at the reversibility of a matrix $M\in \GL (2,\ZZ)$.
A necessary condition clearly is that $M$ and $M^{-1}$ have
the same spectrum (including multiplicities). In other words,
if $P$ is the characteristic polynomial of $M$ with integer
coefficients, it must satisfy the self-reciprocity condition
\begin{equation}\label{eq:self-rec}
   P (x) \, = \, \frac{(-1)^d \ts x^d}{\det (M)} \, 
   P \bigl( \tfrac{1}{x}\bigr) .
\end{equation}
Now, if $d$ is odd or if $\det(M)=-1$, this relation implies
that $1$ or $-1$ is a root, and $P$ is reducible over $\ZZ$.
In particular, $d$ odd and $P$ irreducible immediately
excludes reversibility. This means that, generically, reversible
cases can only occur when $d$ is even and $\det (M) = 1$.

Note that, even if Eq.~\eqref{eq:self-rec} is satisfied, the
reversibility still depends on the underlying integer matrix $M$, and
the class number of $\ZZ[\lambda]$ enters. It is then clear that
deciding on reversibility is a problem that increases with growing
$d$; we refer to the discussion in \cite{BR-toral} for more. However,
for any given characteristic polynomial that is self-reciprocal
according to the condition of Eq.~\eqref{eq:self-rec}, there is at
least one reversible class of matrices, and this can be represented by
the Frobenius companion matrix \cite[Thm.~3]{BR-toral}. \smallskip

A natural extension of symmetries can be considered in the setting
of matrix rings rather than groups, such as $\Mat (d, K)$ instead of
$\GL (d, K)$, where $K$ can itself be a ring (such as $\ZZ$) or a
field (such as $\QQ$). Then, one can define
\[
      \cS (M) \, = \, \{ G \in \Mat (d, K) : [M,G]=0 \ts \} \ts .
\]
Concretely, if $M$ is an integer matrix with irreducible characteristic 
polynomial, and $\lambda$ is any of its roots, one finds $\cS (M)$ to be
isomorphic with an order $\cO$ in the number field $\QQ (\lambda)$ that
satisfies $\ZZ [\lambda] \subseteq \cO \subseteq \cO^{}_{\max}$,
where $\cO^{}_{\max}$ denotes the maximal order in $\QQ (\lambda)$;
see \cite[Ch.~III]{Jacob} as well as \cite[Sec.~3.3]{BR-toral} and references
given there for more.

\subsection{Polynomial automorphisms of the plane}

Let $K$ be a field and consider the group
$\cU^{}_{K} = \GA^{}_{2} (K)$ of polynomial automorphisms of the
affine plane over $K$. Consequently, we have $\XX = K^{2}$ in this
case, which need not be compact.  $\cU^{}_K$ consists of all mappings of
the form
\[
    \begin{pmatrix} x \\ y \end{pmatrix} \, \longmapsto \,
    \begin{pmatrix} P(x,y) \\ Q(x,y) \end{pmatrix}
\]
with $P,Q \in K[x,y]$, subject to the condition that the inverse
exists and is also polynomial. Note that, over general fields,
different polynomials might actually define the same mapping on $K^2$,
but we will distinguish them on the level of the polynomials.

In nonlinear dynamics, where $\GA^{}_{2} (\RR)$ and
$\GA^{}_{2} (\CC)$ have received considerable attention,
a common alternative notation is 
\[
      \xp  \, = \: P(x,y) \, , \quad \yp \, = \: Q(x,y) \ts .
\]
Frequently studied examples include the H\'{e}non quadratic map family,
defined by $P(x,y) = y$ and $Q(x,y) = - \delta \ts x + y^{2} + c$ with
constants $c,\delta \in \CC$ and $\delta \ne 0$. Quite often, for
instance in the context of area-preserving mappings, the starting
point is a polynomial automorphism in \emph{generalised standard form},
\[
     \xp \, = \: x + P^{}_{1} (y) \, , \quad
     \yp \, = \: y + P^{}_{2} (\xp) \ts ,
\]
with \mbox{single{\ts}-}{\nts}variable polynomials $P^{}_{1}$ and
$P^{}_{2}$; compare \cite{GM,RB-poly} and references therein.  Here,
the inverse is simply given by $y = \yp - P^{}_{2} (\xp)$ together
with $x = \xp - P^{}_{1} (y)$.

In a certain sense, such particular normal forms are important, but do
not exhaust the full power of the algebraic setting. Let us explain
this a little in the context of combinatorial group theory.  We begin
by defining three subgroups of $\GA^{}_{2} (K)$ as follows. First,
\[
     \cA \, := \, \big\{ (\bs{a}, M) : \bs{a} \in K^{2}, \,
     M \in \GL (2,K) \big\}
\]
is the group of \emph{affine} transformations, where $(\bs{a},M)$
encodes the mapping $\bs{x} \mapsto M \bs{x} + \bs{a}$. We write
$\bs{a}$ for a column vector, and tacitly identify the elements of
$\cA$ with the canonically corresponding elements of $\GA^{}_{2} (K)$.
Multiplication is defined by
\[
    (\bs{a}, A) (\bs{b}, B) \, = \, (\bs{a} + A \bs{b}, AB) \ts ,
\]
whence $\cA$ is a semi-direct product, namely $\cA = K^{2} \rtimes
\GL (2,K)$. The inverse of an element is $(\bs{a},A)^{-1} =
(-A^{-1} \bs{a}, A^{-1})$.

The second group, $\cE$, is known as the group of \emph{elementary}
transformations. It consists of all mappings of the form
\[
     \begin{pmatrix} x \\ y \end{pmatrix}  \, \longmapsto \,
     \begin{pmatrix} \alpha \ts x + P(y) \\
     \beta \ts y + v \end{pmatrix}
\]
with $P$ a \mbox{single{\ts}-}{\nts}variable polynomial and
$\alpha, \beta, v \in K$ subject to the condition
$\alpha \ts \beta \ne 0$. It is easy to check that the inverse exists
and it of the same form. Transformations of this kind map lines with
constant $y$-coordinate to lines of the same type. It is a well-known
fact that the group $\GA^{}_{2} (K)$ is generated by $\cA$ and $\cE$;
see \cite{Jung,Kulk} as well as \cite[Sec.~1.5]{Wright}.

Finally, our third group, $\cB$, is defined as the intersection
$ \cB \, = \, \cA \cap \cE $, with obvious meaning as subgroups of
$\GA^{}_{2} (K)$. The elements of $\cB$ are called \emph{basic}
transformations, and are mappings of the form
\[
     \begin{pmatrix} x \\ y \end{pmatrix} \, \longmapsto \,
     \begin{pmatrix} \alpha & \gamma \\ 0 & \beta 
     \end{pmatrix} \begin{pmatrix} x \\ y \end{pmatrix}
     + \begin{pmatrix} u \\ v \end{pmatrix}
\]
with $\alpha,\beta,\gamma,u,v \in K$ and $\alpha \ts \beta \ne 0$.
Clearly, also $\cB$ is a semi-direct product,
$\cB = K^{2} \rtimes \cT$, where $\cT$ denotes the subgroups of
$\GL (2,K)$ that consists of all invertible upper triangular matrices
over $K$.

Now, the following result \cite{Wright,Serre} is fundamental to the 
classification of (reversing) symmetries of polynomial automorphisms.

\begin{lem}
  The group\/ $\GA^{}_{2} (K)$ is the free product of the groups\/
  $\cA$ and\/ $\cE$, amalgamated along their intersection, $\cB$,
  which is abbreviated as\/
  $\GA^{}_{2} (K) = \cA \,\raisebox{2pt}{$\underset{\cB}{*}$}\, \cE$.
\end{lem}

Through this result, the problem has been reset in a purely algebraic
way, and one can now explore the subgroup structure \cite{KaSo} of the
amalgamated free product. In particular, one can classify the Abelian
subgroups of $\GA^{}_{2} (K)$, which has trivial centre. Naturally,
$\cS (T)$ for a given $T\in \GA^{}_{2} (K)$ is more complex, and need
no longer be Abelian. When $K$ has characteristic $0$, one can derive
restrictions on the order of other symmetries, which gives access to
the finite subgroups of $\cS (T)$; for details, the reader is
referred to \cite{BR-poly}.

For an important subclass of transformations known as CR elements, one
can say a lot more. In particular, if $K$ is a field of characteristic
$0$, all reversors must be of finite order. If, in addition, the roots
of unity in $K$ are just $\{ \pm 1 \}$, any reversor is an involution
or an element of order $4$, which makes their detection feasible.  The
possible reversing symmetry groups in this case are then the same
three types we saw earlier, in Remark~\ref{rem:types}, for
$2$-dimensional toral automorphisms of infinite order. Since further
details in this setting of combinatorial group theory tend to be a bit
technical, we refer to \cite{BR-poly} and references therein for more.

\section{Shift spaces with faithful $\ZZ$-action}

All examples in the previous section shared the feature that the space
$\XX$ is simple, but the map $T$ on it is not. This is the standard
situation in most dynamical systems that arise from concrete problems,
for instance in nonlinear dynamics. However, it has long been known
\cite{MH} that there is a complementary picture, which arises by
coding orbits in such systems by symbolic sequences, for instance via
itineraries. The latter keep track of a \mbox{coarse{\ts}-}grained
structure in such a way that the full dynamics can be recovered from
them{\ts}---{\ts}at least almost surely in some suitable
\mbox{measure{\ts}-}theoretic sense.

This leads to \emph{symbolic dynamics}, where the space $\XX$ is
`replaced' by a closed shift space $\YY$ (often over a finite
alphabet), and $T$ by the action of the left shift, $S$. More
precisely, one constructs a conjugacy, a semi-conjugacy, or
(typically) a \mbox{measure{\ts}-}theoretic isomorphism that makes the
diagram
\[
\begin{CD}
   \XX @> T >> \XX \\
   @V \phi VV   @VV \phi V \\
   \YY @> S >> \YY
\end{CD}
\]
commutative and $\phi$ as `invertible as possible'. This motivates to
also consider symmetries and reversing symmetries of shift spaces,
where we shall always assume that the action of $\ZZ$ on the shift
space is \emph{faithful} in order to exclude degenerate situations. We
refer to \cite{Kitchens,LM} for general background, and to
\cite{KLP,LPS} for the study of topological Markov chains in this
context.

One immediate problem that arises is the fact that the symmetry group
of a shift space (now called $\XX$ again) is generally huge, in the
sense that it contains a copy of the free group of two generators ---
and is thus not amenable \cite{LM}. This turns a potential
classification into a wild problem, and not much has been done in this
direction.  On the other hand, as has long been known, it is also
possible that one simply gets
$\cS (\XX) = \langle S \ts \rangle \simeq \ZZ$, in which case one
speaks of a \emph{trivial centraliser}, or of a \emph{minimal symmetry
  group}. This is a form of \emph{rigidity}, for which different
mechanisms are possible. Interestingly, rigidity is not a rare
phenomenon \cite{BS}, but actually generic in some sense
\cite{Hoch}, which makes it rather relevant also in practice.

To explore the possibilities a little, let us assume that $\cA$ is a
finite set, called the \emph{alphabet}, and that
$\XX \subseteq \cA^{\ZZ}$ is a closed and shift-invariant set, which
is then automatically compact. Such a space is called a \emph{shift
  space}, or \emph{subshift} for short.

A special role has the `canonical' reversor $R$ defined by
\begin{equation}\label{eq:reflect}
     (R \ts x)^{}_{n} \, := \, x^{}_{-n}
\end{equation}
or any combination of $R$ with a power of the shift $S$. It is clear
that $R$ conjugates $S$ into its inverse on the full shift,
$\XX = \cA^\ZZ$.  More generally, one has the following property.

\begin{lem}
  Let\/ $\XX$ be a shift space with faithful shift action.  If\/ $\XX$
  is reflection-invariant, which means\/ $R (\XX) = \XX$ with the
  mapping\/ $R$ from Eq.~\eqref{eq:reflect}, the system is reversible,
  with\/ $\cR (\XX) = \cS (\XX) \rtimes C^{}_{2}$, where\/
  $C^{}_{2} = \langle R \ts\rangle$.
\end{lem} 

Let us collect a few examples of reversible subshifts, in an informal
manner; see \cite{BRY} and references therein for precise statements
and proofs, and \cite{Coven,Q,AS,P-F,TAO} for general background on
substitution generated subshifts. Among these examples are
\begin{enumerate}
\item the \emph{full} shift \cite{Kitchens,LM}, $\XX = \cA^{\ZZ}$,
  where $\cS (\XX)$ is huge (and not amenable);
\item any \emph{Sturmian} shift \cite{CH}, which is always palindromic
  \cite{DP} and hence reversible, with symmetry group
  $\cS (\XX) \simeq \ZZ$;
\item the \emph{period doubling} shift, defined by the primitive
  substitution rule $a \mapsto ab$, $b \mapsto aa$, again with
  $\cS (\XX) \simeq \ZZ$;
\item the \emph{Thue{\ts}--Morse} (TM) shift, defined by
  $a \mapsto ab$, $b \mapsto ba$, this time with
  $\cS (\XX) \simeq \ZZ \times C^{}_{2}$, where the extra symmetry is
  the letter exchange map defined by $a \leftrightarrow b$;
\item the \emph{square-free} shift, obtained as the orbit closure of
  the characteristic function of the \mbox{square{\ts}-}free integers,
  also with $\cS (\XX) \simeq \ZZ$.
\end{enumerate}

In fact, the last example is quite remarkable, as its rigidity
mechanism relies on the heredity of the shift, as was recently shown
by Mentzen \cite{Mentzen}.  Note that the \mbox{square{\ts}-}free
shift has positive topological entropy, but nevertheless possesses
minimal centraliser. Though this is not surprising in view of known
results from Toeplitz sequences \cite{BK}, it does show that rigidity
as a result of low complexity, as studied in \cite{CQY,CK0,CK,DDMP},
is only {one} of \emph{several} mechanisms. We shall see more in
Section~\ref{sec:multi}.  The \mbox{square{\ts}-}free shift is a
prominent example from the class of $\cB$-free shifts, see
\cite{Mariusz,Bart} and references therein, and also of interest in
the context of Sarnak's conjecture on the statistical independence of
the M\"{o}bius function from deterministic sequences (as discussed at
length in other contributions to this volume).

Of course, things are generally more subtle than in these
examples. First of all, a subshift can be irreversible, as happens for
the one defined by the binary substitution $a \mapsto aba$,
$b \mapsto baa$, where $\cR (\XX) = \cS (\XX) \simeq \ZZ$. Next,
consider the subshift $\XX^{}_{k,\ell}$ defined by the primitive
substitution
\[
       a \, \longmapsto \, a^k \ts b^\ell , \quad b \, \longmapsto \,
       b^k a^\ell
\]
with $k,\ell\in\NN$, which is reversible if and only if $k=\ell$. This
is an extension of the TM shift (which is the case $k=\ell=1$), in the
spirit of \cite{Keane,BGG}.  The symmetry group is
$\cS (\XX^{}_{k,\ell}) \simeq \ZZ \times C^{}_{2}$ in all cases, where
$C^{}_{2}$ is once again the group generated by the letter exchange
map.

Going to larger alphabets,
$\cA = \{ a^{}_{0}, a^{}_{1}, \ldots , a^{}_{N-1} \}$ say, one can
look at a cyclic extension of the TM shift, as defined by the
substitution $a_i \mapsto a_i \ts a_{i+1}$ with the index taken modulo
$N$.  This shift is reflection invariant only for $N=2$, but
nevertheless reversible for any $N$, and even with an involutory
reversor. The symmetry group is $\ZZ \times \nts C^{}_{\! N}$. \smallskip

The quaternary Rudin--Shapiro shift shows another phenomenon.  Its
symmetry group is $\ZZ \times C^{}_{2}$, and it is reversible, but no
reversor is an involution. Instead, there is a reversor of order $4$
(and all reversors have this order), and the reversing symmetry group
is $\ZZ \rtimes C^{}_{4}$, where the square of the generating element
of the cyclic group $C^{}_{4}$ is the extra (involutory) symmetry;
see \cite{BRY} for details on this and the previous examples.

\section{Shift spaces with faithful $\ZZ^d$-action}\label{sec:multi}

It is more than natural to also consider higher-dimensional shift
actions.  Here, given some alphabet $\cA$, a subshift is any closed
subspace $\XX \subseteq \cA^{\ZZ^d}$ that is invariant under the shift
in each of the $d$ directions. With
$\bs{n} = ( n^{}_{1}, \ldots , n^{}_{d})^{T} \in \ZZ^d$ and
$x^{}_{\bs{n}} = \bigl(x^{}_{n^{}_{1}}, \ldots ,
x^{}_{n^{}_{d}}\bigr)$, one defines the shift in direction $i$ by
\[
    (S^{}_{i} \ts x)^{}_{\bs{n}} \, := \, x^{}_{\bs{n} + \bs{e}^{}_{i}} \ts ,
\]
where $\bs{e}_{i}$ is the standard unit vector in direction $i$. The
individual shifts commute with one another, $S_i \ts S_j = S_j S_i$,
for all $1 \leqslant i,j \leqslant d$. Now, we define the
\emph{symmetry group} of $\XX$ as
\[
    \cS (\XX) \, = \, \cent^{}_{\Aut (\XX)} (\cG) \ts ,
\]
where $\cG := \langle S^{}_{1}, \ldots , S^{}_{d} \rangle$ is
a subgroup of $\Aut (\XX)$.

As before, we are only interested in subshifts with faithful shift
action, which means
$\cG = \langle S^{}_{1}\rangle \times \ldots \times \langle S^{}_{d}
\rangle \simeq \ZZ^d$,
where the direct product structure is a consequence of the
commutativity of the individual shifts.  In this case, we define the
\emph{group of extended symmetries} as
\[
    \cR (\XX) \, = \, \norm^{}_{\Aut (\XX)} (\cG) \ts ,
\]
which is the obvious extension of the \mbox{one{\ts}-}dimensional
case. As we shall see shortly, many of the obvious `symmetries' of
$\XX$ are only captured by this extension step.

Unlike before, the structure of the normaliser is generally much
richer now, which also means that $\cR (\XX)$ is a considerably better
topological invariant than $\cS (\XX)$. Indeed, the normaliser can
even be an \emph{infinite} extension of the centraliser when $d>1$, as
can be seen from the full shift as follows; see \cite[Lemma~4]{BRY}.

\begin{fact}\label{fact:full}
  Let\/ $d\in\NN$ and let\/ $\XX = \cA^{\ZZ^d}$ be the full\/
  $d$-dimensional shift over the\/ $($finite or infinite$\ts )$
  alphabet\/ $\cA$. Then, the group of extended symmetries is\/
  $\cR (\XX) = \cS (\XX) \rtimes \GL (d,\ZZ)$.
\end{fact}

The reasoning behind this observation is simple. Each element of
$\cR (\XX)$ must map generators of
$\cG = \langle S^{}_{1}, \ldots , S^{}_{d} \rangle \simeq \ZZ^d$ onto
generators of $\cG$, and thus induces a mapping into $\GL (d,\ZZ)$,
which is the automorphism group of the free Abelian group of rank $d$.
Now, one checks that, for any $M\in\GL (d,\ZZ)$, the mapping
$h^{}_{M}$ defined by
\[
      (h^{}_{M} x )^{}_{\bs{n}} \, = \, x^{}_{M^{-1}\bs {n}} \ts ,
\]
with $\bs{n}$ considered as a column vector, defines an
automorphism of the full shift. This leads to the semi-direct
product structure as stated.

\subsection{Tiling dynamical systems as subshifts}

Substitution tilings of constant block size are a generalisation
of substitutions of constant length, and admit an alternative
description as subshifts, for instance via a suitable symbolic
coding. Classic examples include the chair and the table tiling
\cite{Robbie}, but many more are known \cite{Nat,BG-squiral}. 

Here, we take a look at the chair tiling, which is illustrated in
Figure~\ref{fig:chair}; see \cite{TAO} for more. Its geometric
realisation makes it particularly obvious that any reasonable notion
of a group of full symmetries must somehow contain the elementary
symmetries of the square, simply because the inflation tiling (whose
orbit closure under the translation action of $\ZZ^2$ defines the
tiling dynamical system, with compact space $\XX$) is invariant under
a $4$-fold rotation and a reflection in the horizontal axis.  These
two operations generate a group that is isomorphic with the dihedral
group $D^{}_{4}$, a maximal finite subgroup of $\GL (2,\ZZ)$.

\begin{figure}[t]
\begin{center}
\includegraphics[width=\textwidth]{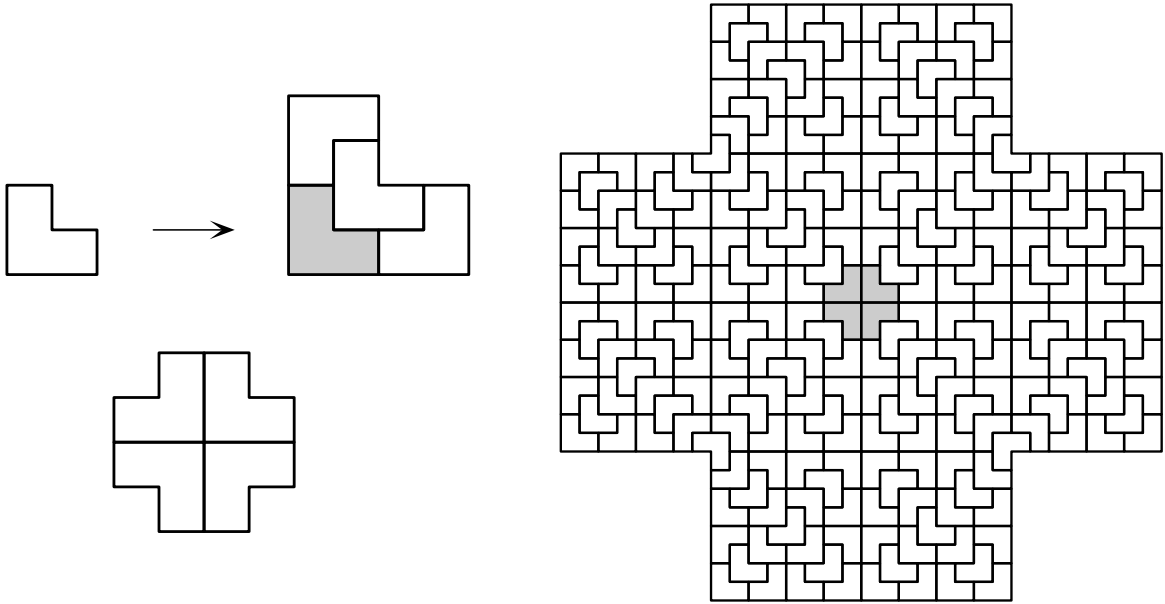}
\end{center}
\caption{The chair inflation rule (upper left panel; rotated tiles are
  inflated to rotated patches), a legal patch with full $D_{4}$
  symmetry (lower left) and a level-$3$ inflation patch generated from
  this legal seed (shaded; right panel). Note that this patch still
  has the full $D_{4}$ point symmetry (with respect to its centre), as
  will the infinite inflation tiling fixed point emerging from
  it.}\label{fig:chair}
\end{figure}

Now, none of these orthogonal transformations occur in the centraliser
of the shift group, which was shown to be minimal in \cite{Olli}. This
is a rigidity phenomenon of \emph{topological} origin, due to the
fibre structure of $\XX$ over its maximal equicontinuous factor
(MEF). Consequently, this example provides ample evidence that one
also needs to consider the normaliser. The general result reads as
follows; see \cite{BRY} for the details.

\begin{thm}
  Let\/ $\XX$ be the hull of the chair tiling, and\/ $(\XX,\ZZ^2)$ the
  corresponding dynamical system. It is topologically conjugate to a
  subshift of\/ $\{ 0,1,2,3\}^{\ZZ^2}$ with faithful shift action.
  Moreover, one has\/ $\cS (\XX) \simeq \ZZ^2$ and\/
  $\cR (\XX) \simeq \ZZ^2 \rtimes D^{}_{4}$, where\/ $D^{}_{4}$ is the
  symmetry group of the square, and a maximal finite subgroup of\/
  $\GL (2,\ZZ)$.
\end{thm}

\begin{proof}[Sketch of proof]
  It is well known that $\XX$ is a.e.\ \mbox{one{\ts}-}to-one over its
  MEF, which is a two-dimensional odometer here.  The orbits of
  non-singleton fibres over the MEF create the topological rigidity
  that enforce the centraliser to agree with the group generated by
  the lattice translations.

  The extension by $D^{}_{4}$ is constructive, via the symmetries of
  the inflation fixed point. Any further extension would require the
  inclusion of a $\GL (2,\ZZ)$-element of infinite order (because
  $D^{}_{4}$ is a maximal finite subgroup of $\GL (2,\ZZ)$), which is
  impossible by the geometric structure (and rigidity) of the
  prototiles.
\end{proof}

Similar results will occur for other tiling dynamical system, also in
higher dimensions. For instance, it is clear that the $d$-dimensional
chair (with $d\geqslant 2$; see \cite{TAO}) will have $\cS = \ZZ^d$
and $\cR = \ZZ^d \rtimes W_d$, where $W_d$ is the symmetry group of
the $d$-dimensional cube, also known as the hyperoctahedral group
\cite{MB}.

Let us note that there is no general reason why the extended symmetry
group should be a semi-direct product (though this will be the most
frequent case to encounter in the applications). In fact, in
(periodic) crystallography, the classification of space groups 
in dimensions $d \geqslant 2$ contains so-called \emph{non-symmorphic}
cases that do not show a semi-direct product structure between
translations and linear isometries \cite{Schw}.  It will be an
interesting question to identify or construct planar shift spaces that
show the planar wallpaper groups as their extended symmetry 
groups.  This and similar results would emphasise once more
that and how the extension from $\cS (\XX)$ to $\cR (\XX)$ is
relevant to capture the full symmetry of faithful shift actions.

\subsection{Shifts of algebraic origin}

There is a particularly interesting and important class of subshifts
that has attracted a lot of attention. They are known as subshifts of
algebraic origin; see \cite{K-book} and references therein. The
important point here is that such a subshift is also an Abelian group
under pointwise addition, and thus carries the correspoondig Haar
measure as a canonical invariant measure.

Here, we take a look at one of the paradigmatic examples
from this class, the Ledrappier shift \cite{Led}. This is the
subshift  $\XX^{}_{\mathrm{L}} \subset \{ 0,1 \}^{\ZZ^2}$
defined as
\begin{equation}\label{eq:L-def}
      \XX^{}_{\mathrm{L}} \, = \, \ker (1 + S^{}_{1} + S^{}_{2} )
      \, = \, \big\{ {x} \in \{0,1\}^{\ZZ^2}\!  : x^{}_{\bs{n}} +
      x^{}_{\bs{n}+\bs{e}^{}_{1}} + x^{}_{\bs{n}+\bs{e}^{}_2}
      = {0} \text{ for all } \bs{n} \in \ZZ^2 \big\} ,
\end{equation}
where the sums are pointwise, and to be taken modulo $2$.  This
definition highlights the special role of elementary lattice
triangles, whose vertices are supporting the local variables that
need to sum to $0$; see Figure~\ref{fig:L} for an illustration.  The
symmetry group is known to be minimal, which can be seen as a rigidity
phenomenon of \emph{algebraic} type. More generally, one has the
following result.

\begin{figure}[t]
\begin{center}
   \includegraphics[width=0.4\textwidth]{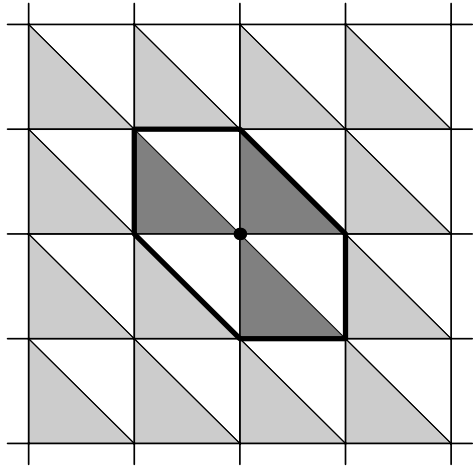}
\end{center}
\caption{Central configurational patch for Ledrappier's shift
  condition, indicating the relevance of the triangular
  lattice. Eq.~\eqref{eq:L-def} implies a condition for the values at
  the three vertices of all elementary $L$-triangles (shaded).  The
  overall pattern of these triangles is preserved by all (extended)
  symmetries. The group $D^{}_{3}$ from Theorem~\ref{thm:L} can now be
  viewed as the colour-preserving symmetry group of the `distorted'
  hexagon as indicated around the origin.  }\label{fig:L}
\end{figure}

\begin{thm}\label{thm:L}
  The symmetry group of Ledrappier's shift\/ $\XX^{}_{\mathrm{L}}$ from
  Eq.~\eqref{eq:L-def} is
\[
    \cS (\XX^{}_{\mathrm{L}}) \, = \,
  \langle S^{}_{1}, S^{}_{2} \rangle \, \simeq \, \ZZ^2,
\]
   while the group of
  extended symmetries is given by
\[
    \cR (\XX^{}_{}) \, = \, \langle S^{}_{1}, S^{}_{2} \rangle
    \rtimes H \, \simeq \, \ZZ^2 \rtimes D^{}_{3} \ts ,
\]  
where\/ $H$ is the group generated by the autormorphisms\/ $h^{}_{A}$
and\/ $h^{}_{B}$, with\/
$A= \left( \begin{smallmatrix} -1 & -1 \\ 1 & 0 \end{smallmatrix}
\right)$ and
$B= \left( \begin{smallmatrix} 0 & 1 \\ 1 & 0 \end{smallmatrix}
\right)$. This group is isomorphic with the dihedral group\/
$D^{}_3 \subset \GL (2,\ZZ)$ that is generated by the corresponding
matrices, $A$ and\/ $B$.
\end{thm}

\begin{proof}[Sketch of proof]
  The triviality of the centraliser is a consequence of the group
  structure, which heavily restricts the homeomorphisms between
  irreducible subshifts that commute with the translations
  \cite{KS,BS,K-book}.

  For the extension to the normaliser, the presence of $D^{}_{3}$ is
  again constructive, and evident from Figure~\ref{fig:L}. One then
  excludes any element of order $6$ that would complete $D^{}_{3}$ to
  $D^{}_{6}$, and finally any element of infinite order that could extend
  the group $D^{}_{3}$. Both types of extensions are impossible
  because any such additional element would change the defining
  condition by deforming the elementary triangles.
\end{proof}

This example is of interest for a number of reasons. First of all, it
shows the phenomenon of rank-$1$ entropy, which is to say that the
number of circular configurations grows exponentially in the
\emph{radius} of the patch, but not in the area. While this means that
the topological entropy still vanishes, Ledrappier's shift is not an
example of low complexity. Second, the spectral structure displays a
mixture of trivial point spectrum with further absolutely continuous
(Lebesgue) components \cite{BW}, which highlights the fact that the
inverse problem of structure determination, in the presence of mixed
spectra, is really a lot more complex than in the case of pure point
spectra. Once again, capturing the full extended symmetry group is 
an important first step in this analysis, as is well-known from
classical crystallography \cite{Schw}.

\subsection{Visible lattice points}

Let us consider the planar point set 
\[
    V \, := \, \{ (x,y)\in\ZZ^2 : \gcd(x,y) =1\} \,\subset\, \ZZ^2 ,
\]
which is known as the set of \emph{visible} (or primitive)
\emph{lattice points}; see the cover page of \cite{Apo} for an
illustration.  The set $V$ has numerous fascinating properties, both
algebraically and geometrically.  In particular, it fails to be a
Delone set, because it has holes of arbitrary size that even repeat
\mbox{lattice{\ts}-}periodically. Nevertheless, the natural density
exists and equals $6/\pi^2 = 1/\zeta(2)$. Moreover, the set $V$ is
invariant under the group $\GL (2,\ZZ)$, which acts transitively on
$V$; see \cite{BMP} and references therein.

The corresponding subshift $\XX^{}_{V}$ is defined as the orbit
closure of the characteristic function $1^{}_V$ under the shift action
of $\ZZ^2$, which turns $(\XX^{}_{V}, \ZZ^2)$ into a topological
dynamical system with faithful shift action and positive topological
entropy. This system, like the \mbox{square{\ts}-}free shift from
above, is \emph{hereditary}, which implies rigidity for the symmetry
group.  On the other hand, due to the way that $\GL (2,\ZZ)$-matrices
act on it, the normaliser is the maximal extension of the centraliser
in this case \cite{BHL}. In fact, there is no reason to restrict to
the planar case here, as the visible lattice points can be defined for
$\ZZ^d$ with any $d\geqslant 2$ (the case $d=1$ gives a finite set
that is not of interest). Thus, one has the following result.

\begin{thm}
  Let\/ $\XX^{}_{V}$ be the subshift defined by the visible lattice
  points of\/ $\ZZ^d$, where\/ $d\geqslant 2$. Then, $\XX^{}_{V}$ has
  faithful shift action with minimal symmetry group,
  $\cS (\XX^{}_{V}) = \ZZ^d$, while the extended symmetry group
  emerges as the maximal extension of it,
  $\cR (\XX^{}_{V}) = \ZZ^2 \rtimes \GL (2,\ZZ)$.
\end{thm}

\begin{proof}[Sketch of proof]
  Here, the triviality of the centraliser, as in the earlier example
  of the \mbox{square{\ts}}-free shift, is a consequence of the
  \emph{heredity} of the subshift \cite{BHL}, and really follows from
  a mild generalisation of Mentzen's approach \cite{Mentzen}.  The
  extension to the normaliser, as explained above, is by all of
  $\GL (d,\ZZ)$, where the semi-direct product structure is the same
  as for the full shift in Fact~\ref{fact:full}.
\end{proof}

More generally, one can study systems of this kind as defined from
primitive lattice systems, for instance in the spirit of
\cite{BHS}. This also covers subshifts that are generated from rings
of integers in general algebraic number fields subject to certain
freeness conditions. This gives a huge class of examples that can be
viewed as multi-dimensional generalisations of $\cB$-free
systems. Interestingly, they are also examples of weak model sets
\cite{BHS}, which gives access to a whole new range of tools from the
interplay of dynamical systems and algebraic number theory
\cite{KR-1,KR-2,Keller,KKL}, in the spirit of the original approach by
Y.~Meyer \cite{Meyer}.

\bigskip
\bigskip

\section*{Acknowledgements}

A substantial part of this exposition is based on joint work (both
past and ongoing) with John Roberts, who introduced me to the concepts
over 25 years ago. More recent activities also profited a lot from the
interaction and cooperation with Christian Huck, Mariusz Lema\'{n}czyk
and Reem Yassawi.  It is a pleasure to thank the CIRM in Luminy for
its support and the stimulating atmosphere during the special program
in the framework of the Jean Morlet semester `Ergodic Theory and
Dynamical Systems in their Interactions with Arithmetic and
Combinatorics', where part of this work was done.


\bigskip
\bigskip


\begin{thebibliography}{99}
\small

\bibitem{AP}
R.L.~Adler and R.~Palais,
Homeomorphic conjugacy of automorphisms of the torus,
\textit{Proc.\ Amer.\ Math.\ Soc.} \textbf{16} (1965) 1222--1225.

\bibitem{AW}
R.L.~Adler and B.~Weiss,
Similarity of automorphisms of the torus,
\textit{Memoirs AMS}, vol.\ 98 (1970),
Amer.\ Math.\ Soc., Providence, RI.

\bibitem{AS}
J.-P.~Allouche and J.~Shallit,
\textit{Automatic Sequences}, Cambridge University Press,
Cambridge (2003).

\bibitem{Apo}
T.M.~Apostol,
\textit{Introduction to Analytic Number Theory},
Springer, New York (1976).

\bibitem{AA}
V.I.~Arnold and A.~Avez,
\textit{Ergodic Problems of Classical Mechanics},
reprint, Addison-Wesley, Redwood City, CA (1989).

\bibitem{MB}
M.~Baake,
Structure and representations of the hyperoctahedral group,
\textit{J.\ Math.\ Phys.} \textbf{25} (1984) 3171--3182.

\bibitem{BGG}
M.~Baake, F.~G\"{a}hler and U.~Grimm,
Spectral and topological properties of a family of generalised
Thue{\ts}--Morse sequences, \textit{J.\ Math.\ Phys.}
\textbf{53} (2012) 032701:1--24;
\texttt{arXiv:1201.1423}.

\bibitem{TAO}
M.~Baake and U.~Grimm,
\textit{Aperiodic Order. Vol.\ $1$: A Mathematical Invitation},
Cambridge University Press, Cambridge (2013).

\bibitem{BG-squiral}
M.~Baake and U.~Grimm,
Squirals and beyond:\ Substitution tilings with singular 
continuous spectrum,
\textit{Ergodic Th.\ \& Dynam.\ Syst.} \textbf{34} (2014)
1077--1102; \texttt{arXiv:1205.1384}.

\bibitem{BGJ}
M.~Baake, U.~Grimm and D.~Joseph, 
Trace maps, invariants, and some of their applications, 
\textit{Int.\ J.\ Mod.\ Phys.} \textbf{B7} (1993) 
1527--1550; \texttt{arXiv:math-ph/9904025}.

\bibitem{BHL}
M.~Baake, C.~Huck and M.~Lema\'{n}czyk,
Positive entropy shifts with small centraliser and large normaliser,
in preparation.

\bibitem{BHS}
M.~Baake, C.~Huck and N.~Strungaru,
On weak model sets of extremal density,
\textit{Indag.\ Math.} \textbf{28} (2017) 3--31; 
\texttt{arXiv:1512.07129}.

\bibitem{BMP}
M.~Baake, R.V.~Moody and P.A.B.~Pleasants,
Diffraction from visible lattice points and $k$-th power 
free integers, 
\textit{Discr.\ Math.} \textbf{221} (2000) 3--42;
\texttt{arXiv:math.MG/9906132}.

\bibitem{BNR}
M.~Baake, N.~Neum\"{a}rker and J.A.G.~Roberts,
Orbit structure and (reversing) symmetries of toral 
endomorphisms on rational lattices,
\textit{Discr.\ Cont.\ Dynam.\ Syst.\ A} \textbf{33} (2013) 527--553;
\texttt{arXiv:1205.1003}.

\bibitem{BR-trace}
M.~Baake and J.A.G.~Roberts, 
Reversing symmetry group of $\GL (2,Z)$ and $\PGL
(2,Z)$ matrices with connections to cat maps and trace maps, 
\textit{J.\ Phys.~A: Math.\ Gen.} \textbf{30} (1997) 1549--1573.

\bibitem{BR-poly}
M.~Baake and J.A.G.~Roberts, 
Symmetries and reversing symmetries of polynomial automorphisms
of the plane, \textit{Nonlinearity} \textbf{18} (2005) 791--816;
\texttt{arXiv:math.DS/0501151}.

\bibitem{BR-toral}
M.~Baake and J.A.G.~Roberts, 
Symmetries and reversing symmetries of toral automorphisms, 
\textit{Nonlinearity} \textbf{14} (2001) R1--R24; 
\texttt{arXiv:math.DS/0006092}.

\bibitem{BR-Bulletin}
M.~Baake and J.A.G.~Roberts.
The structure of reversing symmetry groups, 
\textit{Bull.\ Austral.\ Math.\ Soc.} \textbf{73} (2006) 445--459;
\texttt{arXiv:math.DS/0605296}.

\bibitem{BRW}
M.~Baake, J.A.G.~Roberts and A.~Weiss,
Periodic orbits of linear endomorphisms on the $2$-torus 
and its lattices,
\textit{Nonlinearity} \textbf{21} (2008) 2427--2446;
\texttt{arXiv:0808.3489}.

\bibitem{BRY}
M.~Baake, J.A.G.~Roberts and R.~Yassawi,
Reversing and extended symmetries of shift spaces,
\textit{Discr.\ Cont.\ Dynam.\ Syst.\ A}
\textbf{38} (2018) 835--866;
\texttt{arXiv:1611.05756}. 

\bibitem{BW}
M.~Baake and T.~Ward,
Planar dynamical systems with pure Lebesgue diffraction spectrum,
\textit{J.\ Stat.\ Phys.} \textbf{140} (2010) 90--102;
\texttt{arXiv:1003.1536}.

\bibitem{Bart}
A.~Bartnicka, 
Automorphisms of Toeplitz $\cB$-free systems, 
\textit{preprint} \texttt{arXiv:1705.07021}. 

\bibitem{BS}
S.~Bhattacharya and K.~Schmidt,
Homoclinic points and isomorphism rigidity of algebraic
$\ZZ^{d}$-actions on zero{\ts}-dimensional compact Abelian groups,
\textit{Israel J.\ Math.} \textbf{137} (2003) 189--209.

\bibitem{BS-Book}
Z.I.~Borevich and I.R.~Shafarevich,
\textit{Number Theory},
Academic Press, New York (1966).

\bibitem{BK}
W.~Bulatek and J.~Kwiatkowski,
Strictly ergodic Toeplitz flows with positive entropy and 
trivial centralizers, 
\textit{Studia Math.} \textbf{103} (1992) 133--142.

\bibitem{Coven}
E.M.~Coven,
Endomorphisms of substitution minimal sets,
\textit{Z.\ Wahrscheinlichkeitsth.\ verw.\ Geb.} 
\textbf{20} (1971--1972) 129--133.

\bibitem{CH}
E.M.~Coven and G.A.~Hedlund,
Sequences with minimal block growth,
\textit{Math.\ Systems Th.} \textbf{7} (1973) 138--153.
		
\bibitem{CQY}
E.M.~Coven, A.~Quas and R.~Yassawi,
Computing automorphism groups of shifts using 
atypical equivalence classes, 
\textit{Discrete Analysis} \textbf{2016:3} (28pp);
\texttt{arXiv:1505.02482}.

\bibitem{CK0}
V.~Cyr and B.~Kra,
The automorphism group of a shift of subquadratic growth,
\textit{Proc.\ Amer.\ Math.\ Soc.} \textbf{2} (2016)  613--621;
\texttt{arXiv:1403.0238}.

\bibitem{CK}
V.~Cyr and B.~Kra,
The automorphism group of a shift of linear growth:\ 
beyond transitivity,
\textit{Forum Math.\ Sigma} \textbf{3} (2015) e5 27;
\texttt{arXiv:1411.0180}.

\bibitem{DY}
D.~Damanik, A.~Gorodetski and W.~Yessen,
The Fibonacci Hamiltonian,
\textit{Invent.\ Math.} \textbf{206} (2016) 629--692;
\texttt{arXiv:1403.7823}. 


\bibitem{DDMP}
S.~Donoso, F.~Durand, A.~Maass and S.~Petite,
On automorphism groups of low complexity shifts,
\textit{Ergodic Th.\ \& Dynam.\ Syst.} \textbf{36} (2016) 64--95;
\texttt{arXiv:1501.0051}.

\bibitem{DP}
X.~Droubay and G.~Pirillo,
Palindromes and Sturmian words,
\textit{Theor.\ Comput.\ Sci.} \textbf{223} (1999) 73--85.

\bibitem{Mariusz}
E.H.~El Abdalaoui, M.~Lema\'{n}czyk and T.~de la Rue,
A dynamical point of view on the set of $\cB$-free integers,
\textit{Int.\ Math.\ Res.\ Notices} \textbf{2015} (16) (2015) 7258--7286; 
\texttt{arXiv:1311.3752}.

\bibitem{FKL}
K.~Fr\c{a}czek, J.~Ku\l{}aga and M.~Lema\'{n}czyk,
Non-reversibility and self-joinings of higher orders 
for ergodic flows, \textit{J.\ d'Analyse Math.}
\textbf{122} (2014) 163--227;
\texttt{arXiv:1206.3053}.

\bibitem{Nat}
N.P.~Frank,
Multi-dimensional constant-length substitution sequences,
\textit{Topol.\ Appl.} \textbf{152} (2005) 44--69.

\bibitem{GM}
A.~G\'{o}mez and J.~Meiss,
Reversors and symmetries for polynomial automorphisms
of the complex plane, \textit{Nonlinearity} \textbf{17} (2004)
975--1000; \texttt{arXiv:nlin.CD/0304035}. 

\bibitem{G1}
G.R.~Goodson,
Inverse conjugacies and reversing symmetry groups,
\textit{Amer.\ Math.\ Monthly} \textbf{106} (1999) 19--26.

\bibitem{GLR}
G.~Goodson, A.~del Junco, M.~Lema\'{n}czyk and D.~Rudolph,
Ergodic transformation conjugate to their inverses by involutions,
\textit{Ergodic Th.\ \& Dynam.\ Syst.} \textbf{16} (1996) 97--124.

\bibitem{Hoch}
M.~Hochman,
Genericity in topological dynamics,
\textit{Ergodic Th.\ \& Dynam.\ Syst.}
\textbf{28} (2008) 125--165.

\bibitem{Jacob}
N.~Jacobson,
\textit{Lectures in Abstract Algebra. II. Linear Algebra},
GTM 31, Springer, New York (1953).

\bibitem{Jung}
H.W.E.~Jung,
\"{U}ber ganze irrationale Transformationen der Ebene,
\textit{J.\ Reine Angew.\ Math. (Crelle)} 
\textbf{184} (1942) 161--174.

\bibitem{KaSo}
A.~Karrass and D.~Solitar,
The subgroups of a free product of two groups with an
amalgamated subgroup, \textit{Trans.\ Amer.\ Math.\ Soc.}
\textbf{150} (1970) 227--255.

\bibitem{Kulk}
W.~van der Kulk,
On polynomial rings in two variables,
\textit{Nieuw Arch.\ Wisk.} \textbf{1} (1953) 33--41.

\bibitem{KKL}
S.~Kasjan, G.~Keller and M.~Lema\'{n}czyk
Dynamics of $\cB$-free sets: A view through the window,
\textit{Int.\ Math.\ Res.\ Not.} (2017), in press;
\texttt{arXiv:1702.02375}. 

\bibitem{Keane}
M.~Keane,
Generalized Morse sequences,
\textit{Z.\ Wahrscheinlichkeitsth.\  verw.\ Geb.}  
\textbf{10} (1968) 335--353.

\bibitem{Keller}
G.~Keller,
Generalized heredity in $\cB$-free systems,
\textit{preprint} \texttt{arXiv:1704.04079}. 

\bibitem{KR-1}
G.~Keller and C.~Richard,
Dynamics on the graph of the torus parametrisation,
\textit{Ergodic Th.\ \& Dynam.\ Syst.}, in press;
\texttt{arXiv:1511.06137}. 

\bibitem{KR-2}
G.~Keller and C.~Richard,
Periods and factors of weak model sets,
\textit{preprint} \texttt{arXiv:1702.02383}. 

\bibitem{KLP}
Y.-O.~Kim, J.~Lee and K.K.~Park,
A zeta function for flip systems,
\textit{Pacific J.\ Math.} \textbf{209} (2003) 289--301. 

\bibitem{Kitchens}
B.P.~Kitchens,
\textit{Symbolic Dynamics}, Springer, Berlin (1998).

\bibitem{KS}
B.~Kitchens and K.~Schmidt,
Isomorphism rigidity of irreducible algebraic $\ZZ^{d}$-actions,
\textit{Invent.\ Math.} \textbf{142} (2000) 559--577.

\bibitem{Lamb}
J.S.W.~Lamb,
Reversing symmetries in dynamical systems,
\textit{J.\ Phys.\ A:\ Math.\ Gen.} \textbf{25} (1992) 925--937.

\bibitem{LR}
J.S.W.~Lamb and J.A.G.~Roberts,
Time{\ts}-reversal symmetry in dynamical systems:\ A survey,
\textit{Physica D} \textbf{112} (1998) 1--39.

\bibitem{Led}
F.~Ledrappier,
Un champ markovien peut \^{e}tre d'entropie nulle et m\'{e}langeant,
\textit{C.\ R.\ Acad.\ Sci. Paris S\'{e}r.\ A-B} 
\textbf{287} (1978) A561--A563.

\bibitem{LPS}
J.~Lee, K.K.~Park and S.~Shin,
Reversible topological Markov shifts,
\textit{Ergodic Th.\ \& Dynam.\ Syst.} \textbf{26} (2006) 267--280.

\bibitem{LM}
D.A.~Lind and B.~Marcus,
\textit{An Introduction to Symbolic Dynamics and Coding},
Cambridge University Press, Cambridge (1995).

\bibitem{Mentzen}
M.K.~Mentzen,
Automorphisms of shifts defined by $\mathcal{B}$-free
sets of integers, 
\textit{Coll.\ Math.} \textbf{147} (2017) 87--94.

\bibitem{Meyer}
Y.~Meyer,
\textit{Algebraic Numbers and Harmonic Analysis},
North Holland, Amsterdam (1972).

\bibitem{MH}
M.~Morse and G.A.~Hedlund, 
Symbolic dynamics II.\ Sturmian trajectories,
\textit{Amer.\ J.\ Math.} \textbf{62} (1940) 1--42.

\bibitem{OS}
A.G.~O'Farrel and I.~Short,
\textit{Reversibility in Dynamics and Group Theory},
Cambridge University Press, Cambridge (2015).

\bibitem{Olli}
J.~Olli,
Endomorphisms of Sturmian systems and the discrete chair
substitution tiling system,
\textit{Discr.\ Cont.\ Dynam.\ Syst.\ A}
\textbf{33} (2013) 4173--4186.

\bibitem{Pet}
K.~Petersen,
\textit{Ergodic Theory},
Cambridge University Press, Cambridge (1983).

\bibitem{P-F}
N.~Pytheas Fogg,
\textit{Substitutions in Dynamics, Arithmetics and
Combinatorics}, LNM 1794, Springer, Berlin (2002).

\bibitem{RB}
J.A.G.~Roberts and M.~Baake,
Trace maps as 3D reversible dynamical systems with an
invariant, \textit{J.\ Stat.\ Phys.} \textbf{74} (1994) 829--888.

\bibitem{RB-poly}
J.A.G.~Roberts and M.~Baake, 
Symmetries and reversing symmetries of area-preserving
polynomial mappings in generalised standard form, 
\textit{Physica A} \textbf{317} (2002) 95--112; \newline
\texttt{arXiv:math.DS/0206096}.

\bibitem{Robbie}
E.A.~Robinson,
On the table and the chair,
\textit{Indag.\ Math.} \textbf{10} (1999) 581--599. 

\bibitem{Serre}
J.-P.~Serre, 
\textit{Trees}, 2nd.\ corr.\ printing, Springer, Berlin (2003).

\bibitem{Q}
M.~Queff\'{e}lec,
\textit{Substitution Dynamical Systems{\ts}---{\ts}Spectral Analysis},
LNM 1294, 2nd ed., Springer, Berlin (2010).

\bibitem{K-book}
K.~Schmidt,
\textit{Dynamical Systems of Algebraic Origin},
Birkh\"{a}user, Basel (1995).

\bibitem{Schw}
R.L.E.~Schwarzenberger,
\textit{$N$-dimensional Crystallography},
Pitman, San Francisco (1980).

\bibitem{Sev}
M.B.~Sevryuk,
\textit{Reversible Systems}, LNM 1211, Springer, Berlin (1986).

\bibitem{Wright}
D.~Wright,
Abelian subgroups of $\Aut_{k} (k [X,Y])$ and applications to
actions on the affine plane, \textit{Illinois J.\ Math.}
\textbf{23} (1979) 579--634.

\end{thebibliography}
\end{document}